\documentclass[reqno]{amsart}

\usepackage{amsmath,amsthm,amssymb,amsgen,amsxtra,amsfonts,amsbsy} 
\usepackage[colorlinks,bookmarks=false,linkcolor=blue,linktocpage]{hyperref}
\usepackage[capitalize,nameinlink,noabbrev]{cleveref}
\crefformat{equation}{\textup{#2(#1)#3}}
\usepackage{enumitem}

\title[On Noether's Degree Bound]{On Noether's Degree Bound for Finite Group Schemes}

\author{Gregor Kemper}
\address[GK]{TU M\"unchen, School of CIT, Department of Mathematics (M11), Boltzmannstr.~3, 85748 Garching bei M\"unchen, Germany}
\email{kemper@ma.tum.de}

\author{Christian Liedtke}
\address[CL]{TU M\"unchen, School of CIT, Department of Mathematics (M11), Boltzmannstr.~3, 85748 Garching bei M\"unchen, Germany}
\email{christian.liedtke@tum.de}

\author{Christiane Ott}
\address[CO]{TU M\"unchen, School of CIT, Department of Mathematics (M11), Boltzmannstr.~3, 85748 Garching bei M\"unchen, Germany}
\email{christiane.ott@tum.de}

\subjclass[2020]{13A50, 14L24, 14L15}
\keywords{linear group scheme actions, invariant subring, Noether's degree bound, linearly reductive group scheme}

\newcommand{\NN}{{\mathbb{N}}}
\newcommand{\ZZ}{{\mathbb{Z}}}
\newcommand{\QQ}{{\mathbb{Q}}}
\newcommand{\RR}{{\mathbb{R}}}
\newcommand{\CC}{{\mathbb{C}}}
\newcommand{\GG}{{\mathbb{G}}}
\newcommand{\Spec}{{\rm Spec}\:}

\newcommand{\GL}{{\mathbf{GL}}}
\newcommand{\OO}{{\mathcal O}}
\newcommand{\et}{{\rm \acute{e}t}}
\newcommand{\bmu}{\boldsymbol{\mu}}
\newcommand{\balpha}{\boldsymbol{\alpha}}

\DeclareMathOperator{\diag}{diag}

\theoremstyle{definition}
\newtheorem{Definition}{Definition}[section]
\newtheorem{Example}[Definition]{Example}
\newtheorem{Remark}[Definition]{Remark}
\theoremstyle{plain}
\newtheorem{Theorem}[Definition]{Theorem}
\newtheorem{Proposition}[Definition]{Proposition}

\newtheorem{Lemma}[Definition]{Lemma}

\begin{document}

\begin{abstract}
This paper establishes Noether's classical degree bound $\beta(G)\leq|G|$ for finite and linearly reductive
group schemes. On the other hand, we provide examples of infinitesimal group schemes where $\beta(G)$ is unbounded.
We also generalize Molien's formula to finite and linearly reductive
group schemes.
\end{abstract}

\maketitle

\section{Introduction}

If a finite group $G$ acts linearly on a polynomial ring $k[x_1,...,x_n]$, where~$k$ is a field of
characteristic zero, then a famous theorem of Emmy Noether \cite{Noether} states that
the invariant ring $k[x_1,...,x_n]^G$ can be generated by elements of degree at most the
order $|G|$ of $G$.
This was extended to finite groups whose order is prime to the characteristic of $k$
by Peter Fleischmann \cite{Fleischmann} and John Fogarty \cite{Fogarty}. On the other hand, 
if the characteristic~$p$ of~$k$ divides $|G|$, Noether's bound often fails. For example, David Richman 
\cite{Richman90} proved that for the indecomposable two-dimensional representation $V_2$ of the cyclic group $G = \ZZ/p\ZZ$,
generating the invariant ring $k[m V_2]^G$ of a direct sum of~$m$ copies of $V_2$ requires invariants of a degree
that increases linearly with~$m$. So for $G = \ZZ/p\ZZ$ there exists no degree bound only depending on $G$.

Together with Noether's degree bound, Molien's formula is another classical tool in invariant theory of finite groups
in characteristic~$0$. The formula computes the Hilbert series of the invariant ring without needing to touch a single
invariant. It is well known that Molien's formula extends to the case where the characteristic does
not divide the group order by using a Brauer lift (see \cite[Section~3.4.19]{DK}).

The purpose of this paper is to explore to what extent these results generalize to the case of finite group schemes.
For finite groups, the ``good case'' is the one where the characteristic does not divide the group order. This
arithmetic condition is equivalent to the more conceptual condition that the group be linearly reductive, which
says that every finite-dimensional k-linear representation is a direct sum of irreducible representations.
As it turns out, linear reductivity is the ``correct'' condition when passing from groups to group schemes.

After introducing our notation and making some preparations in \cref{SPrelim}, 
we show in \cref{sNoether} that Noether's bound holds for linearly reductive finite group schemes. This is done
by first considering the connected component of the identity and then the maximal \'etale quotient of $G$.

In many ways the infinitesimal group scheme~$\balpha_p$ forms an analogue to the cyclic group $\ZZ/p\ZZ$. It is also not
linearly reductive. In \cref{sAlpha} we show that for the indecomposable two-dimensional representation $V_2$ of
$G = \balpha_p$, generating the invariant ring $k[m V_2]^G$ of a direct sum of~$m$ copies of $V_2$ requires invariants
of a degree that increases linearly with~$m$. This is the exact analogue of Richman's result mentioned above. In fact,
our proof works more generally for $G = \balpha_q$, with~$q$ a power of~$p$.

Finally in \cref{sMolien} we provide a version of Molien's formula that works for finite linearly reductive group schemes.
This is achieved by a lift to the ring of Witt vectors.

For a general introduction to finite group schemes we refer readers to Michel Brion's 
recent survey article
\cite{Brion}

\section{Preliminaries} \label{SPrelim}

Let~$k$ be a field and let $G\to\Spec k$ be a finite group scheme.
The $k$-algebra $H^0(G,\OO_G)$ carries the structure of a commutative
(but not necessarily co-com\-mu\-ta\-tive), finite-dimensional Hopf algebra over~$k$.
By definition, $|G|:=\dim_k H^0(G,\OO_G)$ is the \emph{order} of $G$ (also called \emph{length} in the context of group schemes).
We also denote the dual Hopf algebra by $k[G]$.
If $G$ is a constant group scheme associated to the finite group $G_{\mathrm{abs}}:=G(k)$
of $k$-rational points, then $k[G]$ is isomorphic to the group algebra $k[G_{\mathrm{abs}}]$.

If $\rho:G\to\GL(V)$ is a finite-dimensional representation,
then we consider the induced
$G$-action on the symmetric $k$-algebra $S(V):=S^\bullet(V)\cong k[x_1,...,x_n]$,
where $n=\dim_k(V)$, and have the invariant subring $S(V)^G\subseteq S(V)$.
We let $\beta(G,\rho)$ - or simply $\beta(\rho)$ if $G$ is clear from the context -
be the minimal number $m$ such that $S(V)^G$
can be generated as a $k$-algebra by homogeneous
invariants of degree $m$.
We also define
$$
\beta(G) \,:=\,\sup_{\rho\,\in\,\mathrm{Rep}_k(G)}\left\{ \beta(\rho) \right\}
\quad\in\quad\ZZ_{\geq0}\cup\{\infty\},
$$
where $\mathrm{Rep}_k(G)$ is the set of all finite-dimensional and $k$-linear representations of $G$.

The following lemma will be used to reduce to the case that~$k$ is algebraically closed.

\begin{Lemma}
\label{lem: reduce to kbar}
 Let $k\subseteq K$ be a field extension.
 Let $G\to\Spec k$ be a finite group scheme and let $\rho:G\to\GL(V)$ be a finite
 dimensional $k$-linear representation.
 If $\rho_K:G_K:=G\times_{\Spec k}\Spec K\to\GL(V\otimes_kK)$
 is the base-change to $K$, then
 $$
   \beta(G,\rho) \,=\, \beta(G_K,\rho_K).
 $$
\end{Lemma}

\begin{proof}
Let $n:=\dim_k(V)$ and then, we have $S:=S^\bullet(V)\cong k[x_1,...,x_n]$,
which we consider as a graded ring with the standard grading, that is, 
$\deg(x_i)=1$ for all $i=1,...,n$.
We have an induced grading on the invariant subring
$R:=S^G\subseteq S$, that is, we find
$R=\bigoplus_{i\geq 0} R_i$ and we have $R_0=k$.
We set $R_+:=\bigoplus_{i\geq1}R_i$, which is an ideal of $R$ 
and define $A(R):=R_+/R_+^2$, which is a finite-dimensional and 
graded $k$-vector space.
It follows from the graded Nakayama lemma that 
$\beta(\rho)$ is equal to the highest degree of an element in $A(R)$.
It is easy to see that there exists an isomorphism of graded $K$-vector
spaces $A(R)\otimes_kK\cong A((S\otimes_kK)^{G_K})$, which implies
that $\beta(G,\rho)=\beta(R)$ coincides with $\beta(G_K,\rho_K)=\beta((S\otimes_kK)^{G_K})$.
\end{proof}

Recall that for a finite group scheme $G\to\Spec k$
there exists a short exact sequence,
called the \emph{connected-\'etale sequence}
$$
1\,\to\,G^\circ\,\to\,G\,\to\,G^{\et}\,\to\,1,
$$
where $G^\circ$ is the connected component of
the identity and $G^\et$ is the maximal \'etale quotient.
If $k$ is perfect, then this sequence splits and 
we obtain an isomorphism $G\cong G^\circ\rtimes G^\et$ (see \cite[Theorem~1.3.13]{Brion}).

If $k\subseteq L$ is a field extension, then a finite group
scheme $G\to\Spec k$ is linearly reductive if and only
if $G_L:=G\times_{\Spec k}\Spec L\to\Spec L$ is linearly
reductive \cite[Proposition 2.4]{AOV}.

Moreover, there is a classification of finite and linearly
reductive group schemes, due to Chin, Hashimoto,
Nagata, and Abramovich--Olsson--Vistoli
(partly independent of each other) 
\cite{Chin, Hashimoto, Nagata, AOV}.
This classification states that $G\to\Spec k$
is linearly reductive if and only if $G^\et$ is of
length prime to $p$ and $G^\circ$ is
of multiplicative type.
In particular,
\begin{enumerate}
\item if $p=0$, then all finite
group schemes over $k$ are linearly reductive.
If $k$ is algebraically closed, then $G$ is the constant
group scheme associated to a finite group.
\item If $k$ is algebraically closed and $p>0$, 
then $G$ is linearly reductive if and only if 
 $G^\et$ is the constant group 
scheme associated to a finite group of order prime to
$p$ and $G^\circ$ is a product of 
group schemes of the form $\bmu_{p^n}$.
\end{enumerate}

\section{Degree bounds} \label{sNoether}

In this section $G$ is a linearly reductive finite group scheme over
a field~$k$ of characteristic~$p > 0$. The goal
is to show that $\beta(G) \le |G|$, so by \cref{lem: reduce to kbar} we may assume
that~$k$ is algebraically closed. By
\cite[Section~(2.10)]{Hashimoto} we have
$G = G^\circ\rtimes G^\et$. We treat $G^\circ$ in the first step and
then $G^\et$. Both are linearly reductive by
\cite[Lemma~2.2]{Hashimoto}.

\begin{Lemma} \label{lG0}%
  In the above situation we have $\beta(G^\circ) \le |G^\circ|$.
\end{Lemma}

\begin{proof}
  As mentioned in \cref{SPrelim} we have
  \[
    G^\circ \cong \bmu_{p^{e_1}} \times \cdots \times \bmu_{p^{e_s}}
  \]
  with $s \ge 0$ and $e_i > 1$. Write $\bmu_{p^{e_i}} = \Spec A_i$
  with $A_i = k[t_i]/(t_i^{p^{e_i}} - 1)$. The representations of
  $\bmu_{p^{e_1}} \times \cdots \times \bmu_{p^{e_s}}$ are well known
  (see, e.g. \cite[Section~2.6]{Hashimoto}): with a suitable choice of
  a basis, a representation $\rho:G^\circ\to\GL(V)$ is given by a
  diagonal matrix
  $\diag(\underline{t}^{\underline{w}_1},\ldots,\underline{t}^{\underline{w}_n})
  \in (A_1 \otimes \cdots \otimes A_s)^{n \times n}$, where
  $\underline{t}^{\underline{w}_i}$ stands for
  $t_1^{w_{i,1}} \cdots t_s^{w_{i,s}}$, and $w_{i,j} \in \ZZ$. Thus
  $G^\circ$ permutes the monomials in $S(V) = k[x_1,\ldots,x_n]$,
  and a polynomial is invariant if and only if all of its monomials
  are invariant. Furthermore, a monomial $x_1^{k_1} \cdots x_n^{k_n}$
  is invariant if and only if
  \begin{equation} \label{eqExponents}%
    \sum_{i=1}^n w_{i,j} k_i \equiv 0 \mod p^{e_j} \quad \text{for
      all} \quad j = 1,\ldots,s.
  \end{equation}
  Now we have the exact same description for the invariants for the
  group $\widetilde G = \ZZ/{p^{e_1}}\ZZ \times \cdots \times\ZZ/{p^{e_s}}\ZZ$
  acting on~$\CC^n$, so it follows from the Noether bound in
  characteristic~$0$ that every exponent vector $(k_1,\ldots,k_n)$
  satisfying~\cref{eqExponents} can be written as a sum of
  exponent vectors satisfying~\cref{eqExponents}, such that in each summand the sum of exponents is
  $\le |\widetilde G|$. A direct argument can be found in
  \cite[Lemma~2.1]{Schmid}, which applies by
  regarding~\cref{eqExponents} as a single equation in the group
  $\ZZ/p^{e_1} \ZZ \oplus \cdots \oplus \ZZ/p^{e_s} \ZZ$.

  So we obtain
  $\beta\bigl(k[x_1,\ldots,x_n]^{G^\circ}\bigr) \le |\widetilde G| =
  |G^\circ|$, and the lemma is proved.
\end{proof}

To prepare for the second step we need the following lemma. The first
part of it can be found in the literature and the second part
should be well known as well. We provide a proof for the convenience of the
reader.

\begin{Lemma} \label{lGN}%
  Let $G$ be an affine group scheme over a scheme $S$ (for the purpose
  of this lemma, $G$ need not be linearly reductive or finite) and let
  $\rho:G\to\GL(V)$ be a finite-dimensional representation. Then
  \begin{enumerate}[label=(\alph*)]
  \item \label{lGNA} If $N \subseteq G$ is a normal subgroup scheme, then the
    $G$-action restricts to the invariants $V^N$.
  \item \label{lGNB} If $G = N \rtimes H$ with $H$ another subgroup scheme, then
    \[
      V^G = (V^N)^H.
    \]
  \end{enumerate}
\end{Lemma}

\begin{proof}
  \begin{enumerate}
  \item[\ref{lGNA}] 
    See \cite[Lemma~16.3]{Waterhouse}.
  \item[\ref{lGNB}] The inclusion $V^G \subseteq (V^N)^H$ is clear since $H$
    and $N$ are subgroup schemes. For the converse, let
    $v \in (V^N)^H$ and consider the morphism $f_v: G \to V$ given by
    applying $G$ to~$v$. Let $T$ be an $S$-scheme. For
    $\sigma \in G(T)$ there exist $\tau \in N(T)$ and $\rho \in H(T)$
    such that $\sigma = \rho \tau$. Therefore
    $f_v(\sigma) = \sigma(v) = \rho(\tau(v)) = v$. So by Yoneda's
    lemma, $f_v$ is the constant morphism mapping everything to~$v$.
    Thus $v \in V^G$. \endproof
  \end{enumerate}
\end{proof}

Going back to our original situation, let $\rho : G \to \GL(V)$ be a finite-dimensional
representation of the linearly reductive finite group scheme $G$ over $k = \bar k$. Then
by Lemma~\ref{lGN}, $G^\et$ acts on $S(V)^{G^\circ}$ and
$S(V)^G = \bigl(S(V)^{G^\circ}\bigr)^{G^\et}$. But the
invariants of $G^\et$ are the same as those of
$\widehat G := G^\et(k)$, which is a group of order is not
divisible by~$p$. So by \cite[Theorem~3.22]{DK} we obtain, using
Lemma~\ref{lG0},
\[
  \beta\bigl(S(V)^G\bigr) \le |\widehat G| \cdot
\beta\bigl(S(V)^{G^\circ}\bigr) \le |\widehat G| \cdot |G^\circ| =
|G^\et| \cdot |G^\circ| = |G|.
\]
So we have shown:

\begin{Theorem}
  For a linearly reductive finite group scheme $G$ we have
  \[
    \beta(G) \le |G|.
  \]
\end{Theorem}

\section{Group schemes that are not linearly reductive}
\label{sAlpha}

There are strong results on the Noether number $\beta(G)$ for finite
groups that are not linearly reductive. In the introduction we
mentioned Richman's \cite{Richman90} result that
$\beta(k[l V_2]^{\ZZ/p\ZZ})$ increases linearly with~$l$, where $l V_2$
denotes the $l$-fold direct sum of the indecomposable two-dimensional
representation of $\ZZ/p\ZZ$ over a field~$k$ of characteristic~$p$. In
fact, later Richman \cite{Richman} showed that this is true in the
much more general situation where $\ZZ/p\ZZ$ is replaced by any finite
group $G$ whose order is divisible by~$p$, and $V_2$ is replaced by
any faithful representation. This implies $\beta(G) = \infty$ for
every finite group that is not linearly reductive.

By letting $G^\circ$ act trivially, this of course extends to group
schemes for which $G^\et$ is not linearly reductive.  However, to the
best of our knowledge no example is known in which an infinitesimal
group scheme $G = G^\circ$ violates Noether's degree bound, or where
we even have $\beta(G) = \infty$.

The following proposition provides such an example by considering the
$l$-fold direct sum of the indecomposable representation $V_2$ of the
infinitesimal group scheme $G = \balpha_q$. Recall that~$\balpha_q$ is
given by the Hopf algebra $H^0(G,\OO_G) = k[t]/(t^q)$, where~$k$ has
characteristic~$p$ and~$q > 1$ is a power of~$p$, with the coproduct
defined by $t \mapsto 1 \otimes t + t \otimes 1$. So the propostion
(or rather, its special case $q = p$) is a direct analog to Richman's
\cite{Richman90} first result. The importance of the infinitesimal
group scheme~$\balpha_p$ is also emphasized by the result \cite[Lemma
2.3]{LRQ} which says that a finite group scheme over an algebraically
closed field of characteristic~$p$ contains a subgroup scheme
isomorphic to $\balpha_p$ or $\ZZ/p\ZZ$. In this sense~$\balpha_p$ and, more
generally, $\balpha_q$ are quintessential examples. 

\newcommand{\pr}{k[x_1, y_1, \dots, x_l, y_l]}
\newcommand{\prg}{k[x_1, y_1, \dots, x_{l+1}, y_{l+1}]}

\begin{Proposition} \label{pAlpha}%
  Consider the following representation
  \begin{align*}
	\rho:
	\left\{ 
	\begin{array} {cl}
		\pr &\to k[t]/(t^q) \otimes \pr\\
		x_i &\mapsto x_i \\
		y_i &\mapsto y_i + tx_i
	\end{array}
	\right.
  \end{align*}
  for the infinitesimal group scheme $\balpha_q$. Then
  \[
    \beta(\balpha_q,\rho) \ge l.
  \]
  In particular, $\beta(\balpha_q) = \infty$.
\end{Proposition}

\begin{proof}
For $0 \leq i \leq q-1$, we consider the polynomials
\begin{align*}
	f_{l,i} := \sum_{i_1 + \dots + i_l= i} 
	x_1^{i_1}y_1^{q-1-i_1} \cdots x_l^{i_l}y_l^{q-1-i_l} 
\end{align*}
as well as
\begin{align*}
	g_{l} := f_{l,q-1}.
\end{align*}
In order to see that $g_l$ is an invariant, we first consider the following equation that we will prove via induction:
For every $0 \leq i \leq q-1$, we have
\begin{align}
f_{l,i} = \sum_{j=i}^{q-1} 
\binom{j}{i} t^{j-i}  \rho(f_{l, j}). \label{lem_alpha_q}
\end{align}
To prove this, we will make use of the identity
 \begin{align}
	\sum_{r=0}^{n} \binom{r}{{x}}\binom{n-r}{y} = \binom{n+1}{x+y+1} \label{chu_vand}
\end{align}
for every $x,y,n \in \NN$, see \cite[Proposition 3.32(f)]{Grinberg}. 
We first consider the case where $l = 1$. For every $0 \leq j \leq q-1$, we have $f_{1,j} = x_1^j y_1^{q-1-j}$. Thus,
\begin{align*}
	&\sum_{j=i}^{q-1} 
	\binom{j}{i} t^{j-i}  \rho(f_{1, j})\\
    = &\sum_{j=i}^{q-1} 
	\binom{j}{i} t^{j-i}  \rho \left( x_1^j y_1^{q-1-j} \right)
    \\
	= &\sum_{j=i}^{q-1} 
	\binom{j}{i} t^{j-i}  x_1^j (y_1+ tx_1)^{q-1-j}  \\
	= & \sum_{j=i}^{q-1} 
	\binom{j}{i} t^{j-i}  x_1^j 
	\sum_{s=0}^{q-1-j} \binom{q-1-j}{s} y_1^s t^{q-1-j-s}x_1^{q-1-j-s}
    \\
	= & \sum_{j=i}^{q-1} 
	\binom{j}{i}
		\sum_{s=0}^{q-1-j}
		\binom{q-1-j}{s} y_1^s t^{q-1-s-i}x_1^{q-1-s}
	 \\
    = & 
	\sum_{s=0}^{q-1-i} 
	\sum_{\substack{j=i }}^{q-1-s} 
	\binom{j}{i}
	\binom{q-1-j}{s} y_1^s t^{q-1-s-i}x_1^{q-1-s}\\
	= & 
	\sum_{s=0}^{q-1-i} 
	x_1^{q-1-s} y_1^st^{q-1-s-i}
	\sum_{j=0}^{q-1} 
		\binom{j}{i}
		\binom{q-1-j}{s} \\
	\overset{\cref{chu_vand}}{=} & 
	\sum_{s=0}^{q-1-i} 
	x_1^{q-1-s} y_1^s t^{q-1-s-i}
	\binom{q}{i+s+1}
	\overset{\text{Lucas's Theorem}}{=}  x_1^i y_1^{q-1-i} 
	=  f_{1,i}.
\end{align*}

For the induction step $l \to l+1$, let $0 \leq i \leq q-1$. We need to show that
\begin{align}
	f_{l+1,i}
	= & \sum_{j=i}^{q-1} 
	\binom{j}{i} t^{j-i}  \rho(f_{l+1, j}). \label{lemma_ind}
\end{align}
To do so, we first simplify both sides. This leads to 
\begin{align*}
	f_{l+1,i}
	= & \sum_{s= q-1-i}^{q-1} x_{l+1}^{q-1-s} y_{l+1}^s f_{l, i-(q-1-s)} 
\end{align*}
as well as
\begin{align*}
	& \sum_{j=i}^{q-1} 
		\binom{j}{i} t^{j-i}  \rho(f_{l+1, j})\\
	= & \sum_{j=i}^{q-1} 
		\binom{j}{i} t^{j-i}  \rho 
		\left(
			\sum_{r= 0}^{j} x_{l+1}^r y_{l+1}^{q-1-r} f_{l, j-r}
		\right) \\
	= & \sum_{j=i}^{q-1} 
		\binom{j}{i} t^{j-i}  
			\sum_{r= 0}^{j} x_{l+1}^r (y_{l+1}+tx_{l+1})^{q-1-r} \rho (f_{l, j-r}) \\	
	= & \sum_{j=i}^{q-1} 
		\binom{j}{i} t^{j-i}  
			\sum_{r= 0}^{j} x_{l+1}^r
				\sum_{s=0}^{q-1-r}
				\binom{q-1-r}{s}
				y_{l+1}^s t^{q-1-r-s} x_{l+1}^{q-1-r-s}
				\rho (f_{l, j-r})
			\\	
	= & \sum_{j=i}^{q-1} 
		\binom{j}{i} 
		\sum_{r= 0}^{j} 
			\sum_{s=0}^{q-1-r}
			\binom{q-1-r}{s}
			y_{l+1}^s t^{q-1-r-s+j-i} x_{l+1}^{q-1-s}
			\rho (f_{l, j-r}) \\
    =  & 
	\sum_{s=0}^{q-1}
		\sum_{j=i}^{q-1}  
			\binom{j}{i} 
			\sum_{r= 0}^{
				\substack{
					\min \{ j,
				 	q-1-s \}}				
			} 
				\binom{q-1-r}{s}
				y_{l+1}^s t^{q-1-r-s+j-i} x_{l+1}^{q-1-s}
				\rho (f_{l, j-r})
	\\
	 = & \sum_{s=0}^{q-1}
		x_{l+1}^{q-1-s} y_{l+1}^s 
		\sum_{j=i}^{q-1} 
			\binom{j}{i} 
			\sum_{r= 0}^{
				\substack{
					\min \{ j,
				 	q-1-s \}}
			} 
				\binom{q-1-r}{s}
				t^{q-1-r-s+j-i}
				\rho (f_{l, j-r}).
\end{align*}
Comparing coefficients, it suffices to show that
\begin{align}
	\nonumber
	&\sum_{j=i}^{q-1} 
		\binom{j}{i} 
		\sum_{r= 0}^{\min \{ j, q-1-s \}} 
			\binom{q-1-r}{s}
			t^{q-1-r-s+j-i} 
		\rho (f_{l, j-r})
	\\
	= &
	\left\{ 
		\begin{array} {cl}
			f_{l, i-(q-1-s)}  & \text{for every } q-1-i \leq s \leq q-1 \\
			0 &\text{for every } 0 \leq s < q-1-i . \label{bed_2} 
		\end{array}
	\right. 
\end{align}
We simplify the first line and we obtain
\begin{align}
	\nonumber &\sum_{j=i}^{q-1} 
		\binom{j}{i} 
			\sum_{r= 0}^{\min \{ j, q-1-s \}} 
				\binom{q-1-r}{s}
				t^{q-1-r-s+j-i}
			\rho (f_{l, j-r})
	\\
    \nonumber = & 
	\sum_{\beta=\max \{ 0, i-q+1+s \}}^{q-1}
		t^{\beta-i+q-1-s} 
		\rho (f_{l, \beta})
		\sum_{j= \max \{ i, \beta \}}^{\min \{q-1, \beta + q-1-s \} } 
			\binom{j}{i} 
			\binom{q-1 +\beta-j}{s}
	\\
	\nonumber \overset{\text{Luc.}}{=} & 
	\sum_{\beta=\max \{ 0, i-q+1+s \}}^{q-1}
		t^{\beta-i+q-1-s} 
		\rho (f_{l, \beta})
		\sum_{j= 0}^{q-1+\beta } 
			\binom{j}{i} 
			\binom{q-1 +\beta-j}{s}
		\\	
	\overset{\cref{chu_vand}}{=}  & 
	\sum_{\beta=\max \{ 0, i-q+1+s \}}^{q-1}
		t^{\beta-i+q-1-s} 
		\rho (f_{l, \beta})
		\binom{q+\beta}{i+s+1} .
	\label{eq_t1}
\end{align}
For every $0 \leq s < q-1-i$, this sum is equal to zero by Lucas's Theorem. Hence, the second condition in \cref{bed_2} is satisfied. Now, it only remains to show the first condition. For this purpose, we consider $q-1-i \leq s \leq q-1$. Then, 
\begin{align}
	\nonumber  f_{l, i-(q-1-s)} 
	\overset{\text{induction}}{=}  &
	\sum_{j=i-q+1+s}^{q-1} 
		\binom{j}{i-q+1+s} t^{j-i+q-1-s}  \rho(f_{l, j})\\	
	= &	
	\sum_{j=\max \{ 0, i-q+1+s \}}^{q-1} 
	\binom{j}{i-q+1+s} t^{j-i+q-1-s}  \rho(f_{l, j}). \label{eq_t2}
\end{align}
Using Lucas's Theorem, we see that \cref{eq_t2} is equal to \cref{eq_t1}. As a result, the first condition in \cref{bed_2} follows. Hence, we have shown \cref{lemma_ind}. 

Now, we obtain
\begin{align*}
	g_{l} = f_{l,q-1} 
	\overset{\cref{lem_alpha_q}}{=} \sum_{j=q-1}^{q-1} 
	\binom{j}{q-1} t^{j-(q-1)}  \rho(f_{l, j})
	= \rho(f_{l, q-1})
	= \rho(g_{l}).
\end{align*}
In particular, $g_l$ is an invariant.

Next, we consider a decomposition $g_l = \sum_{r=1}^{R} \underbrace{\prod_{s=1}^{S_r} h_{r,s}}_{=: h_r}$ of $g_l$ into invariants such that $\deg(g_l) \geq \deg(h_{r,s})$.
Without loss of generality, we might make the following assumptions:

\begin{enumerate}
	\item \label{ges_deg} Every $h_r$ is homogeneous with respect to its total degree. Moreover, we have $\deg h_r = l(q-1)$: \\
	Every $h_{r,s}$ can be written as a sum of its homogeneous components $h_{r,s}^{(d)}$ with respect to $\deg$.
	Since $\rho(h_{r,s}) = \rho \left(\sum_d h_{r,s}^{(d)}\right) = \sum_d \rho \left(h_{r,s}^{(d)} \right)$ and $x_i \mapsto x_i, y_i \mapsto y_i + tx_i$, it follows that $\rho(h_{r,s}^{(d)})$ is the homogeneous component of degree $d$ (in $x_1, y_1, \dots, x_l, y_l$) of $\rho(h_{r,s})$. In particular, every $h_{r,s}^{(d)}$ is again an invariant. Hence, we can replace $h_r$ by suitable alternatives that are homogeneous with respect to $\deg$ and have degree $l(q-1)$.	
	
	\item \label{ges_hom_y} Every $h_r$ is homogeneous with respect to $\deg_{y_1, \dots, y_{l}}$. Furthermore, we have $\deg_{y_1, \dots, y_{l}} h_r = (l-1)(q-1)$: \\
	The argument is similar to \eqref{ges_deg} when we consider $\deg_{y_1, \dots,y_{l}, t}$.
	\item \label{ges_deg_X} Every $h_r$ is homogeneous with respect to $\deg_{x_1, \dots, x_{l}}$. Furthermore, we have $\deg_{x_1, \dots, x_{l}} h_r = q-1$. \\
	This follows from \eqref{ges_deg} and \eqref{ges_hom_y}.
	\item \label{ges_hom_xVyV} Every $h_{r}$ is homogeneous with respect to $\deg_{x_i, y_i}$ for every $1 \leq i \leq l$. Furthermore, $\deg_{x_i, y_i}h_r = q-1$.\\
	The argument is similar to \eqref{ges_deg} when we consider $\deg_{x_i, y_i}$.
	\item \label{sing_hom_xVyV} Every $h_{r,s}$ is homogeneous with respect to $\deg_{x_i, y_i}$ for every $1 \leq i \leq l$.\\
	The argument is similar to \eqref{ges_deg} when we consider $\deg_{x_i, y_i}$.
	\item \label{sing_not_const} Every $h_{r,s}$ is non-constant. \\
	Otherwise, we would be able to combine it with an additional factor by \eqref{ges_deg}.
	\item \label{ges_deg_y} Every $h_{r}$ satisfies $\deg_{y_i} h_{r}  \leq q-1 $ for every $1 \leq i \leq q-1$.\\
	This follows from $\deg_{y_i} h_{r} \leq \deg_{x_i,y_i} h_{r} \overset{\eqref{ges_hom_xVyV}}{=} q-1 $.
	\item \label{sing_deg_y} Every $h_{r,s}$ satisfies $\deg_{y_i} h_{r,s} \leq q-1 $ for every $1 \leq i \leq q-1$.\\
	This follows from \eqref{ges_deg_y} since $\sum_{s=1}^{S_r} \deg_{y_i} h_{r,s} =\deg_{y_i} h_{r} \leq q-1 $.
	\item \label{sing_deg_x} Every $h_{r,s}$ satisfies $\deg_{x_1, \dots, x_{l}} h_{r,s} >0$:\\
	Assume that $\deg_{x_1, \dots, x_{l}} h_{r,s} =0$. Because of \eqref{sing_not_const}, we might assume, without loss of generality, that $\deg_{y_1} h_{r,s}>0$. 
	Now, \eqref{sing_hom_xVyV} and \eqref{sing_deg_y} allow us to decompose $h_{r,s}$ as $h_{r,s} = \sum_{i=0}^w y_1^ix_1^{L-i} b_i$ for suitable $w,L \in \NN$ such that $1 \leq w \leq q-1$, $b_i \in k[x_2, y_2, \dots, x_l, y_l]$ and $b_w \neq 0$. Using $\deg_{x_1, \dots, x_{l}} h_{r,s} =0$, we obtain $h_{r,s} = y_1^w b_w$. Since $h_{r,s}$ is an invariant, it follows that
	\begin{align*}
		y_1^w b_w
		&= h_{r,s}
		= \rho \left( h_{r,s} \right)
		= \rho \left( y_1^w b_w\right)\\
		&= (y_1+tx_1)^w \rho(b_w)
		= \left(
		\sum_{j=0}^w \binom{w}{j} y_1^j (tx_j)^{w-j}
		\right) \rho(b_w).
	\end{align*}
	Comparing coefficients for $j=w$, we obtain $b_w = \rho(b_w)$. Comparing coefficients for $j \neq w$ leads to $b_w=0$. This is a contradiction.	
\end{enumerate}
Now, we obtain 
\[
q-1 
\overset{\eqref{ges_deg_X}}{=} \deg_{x_1, \dots, x_{l}} h_r
= \deg_{x_1, \dots, x_{l}} \prod_{s=1}^{S_r} h_{r,s}
= \sum_{s=1}^{S_r} \deg_{x_1, \dots, x_{l}} h_{r,s}
\overset{\eqref{sing_deg_x}}{\geq} \sum_{s=1}^{S_r} 1
= S_r
\]
for every $1 \leq r \leq R$. Thus, 
\begin{align*}
	l(q-1) 
	&\overset{\eqref{ges_deg}}{=}\deg h_r
	= \sum_{s=1}^{S_r} \deg h_{r,s}\\
	&\leq S_r \cdot \max_{1 \leq s \leq S_r}\{ \deg h_{r,s} \}
	\leq (q-1) \cdot \max_{1 \leq s \leq S_r} \deg h_{r,s} ,
\end{align*}
which leads to $l \leq \max_{1 \leq s \leq S_r} \deg h_{r,s}$.
In particular, we have $\beta(\balpha_q,\rho) \ge l$.
\end{proof}

\section{Molien's formula} \label{sMolien}

The main goal of this section is to generalize Molien's formula
to the case of a finite linearly reductive group scheme. This is done
via a lifting to the ring of Witt vectors, which we describe now.

If $k$ is algebraically closed and $p>0$, then we let
$W(k)$ be the ring of Witt vectors over $k$, which is a complete
DVR of characteristic zero with residue field $k$ and 
whose maximal ideal is generated by $p$.
If $G\to\Spec k$ is a finite and linearly reductive group
scheme, then there exists a flat 
lift $\mathcal{G}\to\Spec W(k)$, the \emph{canonical lift},
which is almost
unique, see \cite[Proposition 2.4]{LRQ}.
Let $K$ be the field of fractions of $W(k)$ and
let $\overline{K}$ be an algebraic closure of $K$.
Let $\mathcal{G}_{\overline{K}}:=\mathcal{G}\times_{\Spec W(k)}\Spec\overline{K}$
be the geometric generic fibre, which is a constant group
scheme over $\overline{K}$.
Then, still assuming $k$ to be algebraically closed,
then every finite-dimensional $k$-linear representation lifts
(uniquely up to isomorphism) to a $W(k)$-linear representation of $\mathcal{G}$.
From this, we obtain a specialisation map
$$
\mathrm{sp}\,:\,\mathrm{Rep}_{\overline{K}}(\mathcal{G}_{\overline{K}})
\,\to\,\mathrm{Rep}_k(G)
$$
of finite dimensional $k$-linear representations of $G$
and finite-dimensional $\overline{K}$-linear representations 
$\mathcal{G}_{\overline{K}}$.
The map $\mathrm{sp}$ respects degrees and simplicitiy and 
it is compatible with direct sums, tensor products, and duals,
see \cite[Proposition 2.9]{LRQ}.

If $G$ is a finite and linearly reductive group scheme
over a field $k$, we can define the 
\emph{abstract group} $G_{\mathrm{abs}}$
associated to $G$, see \cite[Definition 2.2]{LRQ}.
\begin{enumerate}
\item If $p=0$, then $G_{\mathrm{abs}}\cong G(\overline{k})$.
\item If $p>0$, then $G_{\mathrm{abs}}\cong \mathcal{G}'_{\overline{K}}(\overline{K})$,
where $\mathcal{G}'\to\Spec W(\overline{k})$ is the canonical lift
of $G_{\overline{k}}$.
\end{enumerate}
If $k$ is algebraically closed, then we can use the specialisation map $\mathrm{sp}$
to obtain another specialisation map
$$
\mathrm{sp}_\CC\,:\,\mathrm{Rep}_\CC(G_{\mathrm{abs}}) \,\to\,
\mathrm{Rep}_k(G)
$$
that respects degrees and simplicitiy and 
it is compatible with direct sums, tensor products, and duals,
see \cite[Corollary 2.11]{LRQ}.
In fact, this can be descended to $\overline{\QQ}$ or
even $\QQ(\zeta_{|G|})$, where $\zeta_{|G|}$ is a primitive $|G|$.th
root of unity.

\subsection{Lifting the invariant subring}
By the following proposition, we can also lift the invariant ring.

\begin{Proposition}
\label{prop: lift}
 Let $k$ be an algebraically closed field of characteristic $p>0$,
 let $G\to\Spec k$ be a finite and linearly reductive group scheme, 
 let $\rho:G\to\GL(V)$ be an $n$-dimensional representation, 
 and consider $S:=S^\bullet(V)\cong k[x_1,...,n]$ with the induced $G$-action.
 Let $\mathcal{G}\to\Spec W(k)$ be the canonical lift of $G$
 and let $\widetilde{\rho}:\mathcal{G}\to\GL(\mathcal{V})$ be the 
 lift of $\rho:G\to\GL(V)$.
 \begin{enumerate}
 \item
 We set
 \begin{equation}
 \label{eq: molien lift}
  \mathcal{S}_n\,:=\,S^n(\mathcal{V})\mbox{ \qquad and \qquad } 
  \mathcal{S} \,:=\, \bigoplus_{n\geq0}\mathcal{S}_n\,\cong\, W(k)[x_1,...,x_n]
 \end{equation}
 which turns $\mathcal{S}$ into a graded $W(k)$-algebra with a $\mathcal{G}$-action.
\item We obtain a commutative diagramme of graded $W(k)$-algebras
 $$
 \begin{array}{ccc}
  \mathcal{S} &\to& S\\
  \cup&&\cup\\
   \mathcal{S}^{\mathcal{G}} &\to&S^G,
   \end{array}
 $$
 where the horizontal arrows are reduction modulo $p$, where the 
 $\mathcal{G}$-action and the $G$-action in the top row are compatible,
 and where $\mathcal{S}^{\mathcal{G}}$ is a flat $W(k)$-algebra.
 In particular, $\mathcal{S}^G$ is a flat lift of $S^G$ over $W(k)$.
 \item Finally, we have
 $$
\dim_K S^n(\mathcal{V}_K)^{\mathcal{G}_K} \,=\, \dim_k  S^n(V)^G.
 $$
 \end{enumerate}
 \end{Proposition}

\begin{proof}
For each integer $n\geq0$, we have a commutative diagramme
of $W(k)$-algebras 
$$
   \begin{array}{ccccc}
    S^n(\mathcal{V}\otimes_{W(k)}K)&\supseteq& S^n(\mathcal{V})&\to& S^n(V)\\
    \cup&& \cup && \cup\\
    S^n(\mathcal{V}\otimes_{W(k)}K)^{\mathcal{G}_K}&\supseteq& S^n(\mathcal{V})^{\mathcal{G}}&\to& S^n(V)^G\\
    \cup&& \cup && \cup \\
    K &\leftarrow& W(k) &\to& k
   \end{array}
$$
that is compatible with the natural actions of $\mathcal{G}_K$, $\mathcal{G}$, and $G$.
Thus, $\mathcal{S}$ is a graded $W(k)$-algebra with a $\mathcal{G}$-action, we have
$$
\bigoplus_{n\geq0} S^n(\mathcal{V})\,=\,S^\bullet(\mathcal{V}) \,\cong\,
W(k)[x_1,...,x_n],
$$
and it maps surjectively onto $S^\bullet(V)\cong k[x_1,..,x_n]$.
The kernel of this surjection is the ideal generated by $p$.

Being a $W(k)$-submodule of $\mathrm{Sym}^n(\mathcal{V})$, it follows
that $S^n(\mathcal{V})^{\mathcal{G}}$ is torsion-free as a $W(k)$-module.
Using properties of the specialisation map $\mathrm{sp}$, we obtain
the equality claimed in \cref{eq: molien lift},
which implies that $S^n(\mathcal{V})^{\mathcal{G}}$ is a flat $W(k)$-module.
\end{proof}

\begin{Remark}
If $k$ is an algebraically closed field of characteristic $p>0$ and
if $G\to\Spec k$ is a finite group scheme that is not linearly reductive, 
then a lift of $G$ over $W(k)$ or even over an extension of $W(k)$
may not exist:
See \cite[Introduction, Example (-B)]{MO} and \cite[page 266]{Oort}
for an example.
\end{Remark} 

\subsection{Reynolds operator}
A central tool in invariant theory are  Reynolds operators
and the following shows that we can even find such an operator
that is compatible with lifts.

\begin{Proposition}
 With assumptions and notations as in Proposition \ref{prop: lift},
 there exists a surjective
 homomorphism of $\mathcal{S}^{\mathcal{G}}$-modules
 $$
  \pi \,:\, \mathcal{S}\,\to\,\mathcal{S}^{\mathcal{G}}
 $$
 that splits the inclusion $\mathcal{S}^{\mathcal{G}}\to\mathcal{S}$
 as $\mathcal{S}^{\mathcal{G}}$-modules.
 In particular, $\mathcal{S}^{\mathcal{G}}$ is a direct summand
 of $\mathcal{S}$ as a $\mathcal{S}^{\mathcal{G}}$-module.
\end{Proposition}

\begin{proof}
Being linearly reductive, we can decompose the $G$-module
$S^n(V)$ into isotypical components for the simple $G$-representations.
This yields a canonical projection operator $S^n(V)\to S^n(V)^G$
and thus, a canonical direct sum decomposition 
$S^n(V)\cong S^n(V)^G\oplus W_n$.
Using the specialisation morphism $\mathrm{sp}$, we obtain a 
canonical direct sum decomposition
$$
S^n(\mathcal{V}) \,\cong\, S^n(\mathcal{V})^{\mathcal{G}}\,\oplus\,\mathcal{W}_n,
$$
which is $\mathcal{G}$-stable,
and a canonical projection $\pi_n:S^n(\mathcal{V})\to S^n(\mathcal{V})^{\mathcal{G}}$,
which is $\mathcal{G}$-equivariant (w.r.t. to the trivial operation on the right).
Taking direct summands, we obtain a direct sum decomposition
$\mathcal{S}=\mathcal{S}^{\mathcal{G}}\oplus W_\bullet$ of $W(k)$-modules, 
where $W_\bullet:=\bigoplus_{n\geq0} W_n$.
Moreover, this decomposition also respects
$\mathcal{G}$-representations, which implies
that $W_\bullet$ is a $\mathcal{S}^{\mathcal{G}}$-module.
\end{proof}

\subsection{Molien series}
If a group scheme $G\to\Spec k$ acts linarly 
on a polynomial ring $S=k[x_1,...,x_n]$,
then we define the \emph{Molien series} to be
$$
 \mathrm{Molien}(S^G,t) \,:=\,
 \sum_{i=0}^\infty  \dim_k(S^G_i)\cdot t^i \,\in\, \ZZ[[t]],
$$
where $S^G_i\subseteq S^G\subseteq S\cong k[x_1,...,x_n]$ 
denotes the $k$-vector space of polynomials of degree $i$ 
of $S^G$.

Being linear, the $G$-action on $S$ is induced from the $k$-linear
representation $G\to\GL(S_1)\cong\GL_{n,k}$.
If $G\to\Spec k$ is linearly reductive, then we have 
the abstract group $G_{\mathrm{abs}}$, 
a representation $\rho_\CC:G_{\mathrm{abs}}\to\GL_n(\CC)$,
and we define the \emph{abstract Molien series} to be
$$
 \mathrm{Molien}_{\mathrm{abs}}(S^G,t) \,:=\,
 \sum_{i=0}^\infty  \dim_k(\CC[x_1,...,x_n]^{G_{\mathrm{abs}}}_i)\cdot t^i \,\in\, \ZZ[[t]].
$$
Combining Proposition \ref{prop: lift} with Molien's
theorem \cite{Molien}, see, for example, \cite[Section 3.4]{DK} for a modern account, 
we obtain the following.

\begin{Theorem}[Molien formula]
\label{thm: molien}
Let $k$ be a field of characteristic $p\geq0$ and let
$\overline{k}$ be an algebraic closure of $k$.
Let $G\to\Spec k$ be a finite group scheme, let
$G\to\GL(V)$ be an $n$-dimensional representation, and consider the
induced $G$-action on $S:=S(V)\cong k[x_1,..,n]$.
 \begin{enumerate}
 \item  If $p=0$, then 
 $$
  \mathrm{Molien}(S^G,t) \,=\, 
  \frac{1}{|G|}\,
  \sum_{g\in G(\overline{k})} \frac{1}{\det(1-g^{-1}t,V_{\overline{k}})} 
 $$
 and it coincides with $\mathrm{Molien}_\CC(\CC[x_1,...,x_n]^G,t)$.
 \item
 If $p>0$, then assume moreover that $G\to\Spec k$ is linearly reductive.
 \begin{enumerate}
 \item We have
 $$
 \mathrm{Molien} (S^G,t) \,=\, 
 \mathrm{Molien}( S^G\otimes_k\overline{k},t) \,=\,
 \mathrm{Molien}( (S\otimes_k\overline{k})^{G_{\overline{k}}},t).
 $$
 \item 
 If $k$ is moreover algebraically closed, then we have in
  the notations from Proposition \ref{prop: lift} 
 $$
 \mathrm{Molien}(S^G,t) \,=\, \mathrm{Molien}(\mathcal{S}_K^{\mathcal{G}_K},t)
 $$
 and the latter has been determined in (1).
 \end{enumerate}
 \end{enumerate}
In both cases, the Molien series coincide with
 \begin{eqnarray*}
  \mathrm{Molien}_{\mathrm{abs}}(S^G,t) &=&
  \mathrm{Molien} (\CC[x_1,...,x_n]^{G_{\mathrm{abs}}},t) \\
  &=& \frac{1}{|G_{\mathrm{abs}}|}\,
  \sum_{g\in G_{\mathrm{abs}}} \frac{1}{\det(1-g^{-1}t,\CC^n)}.
 \end{eqnarray*}
\end{Theorem}

\begin{proof}
We have $(S^G)\otimes_k\overline{k}\cong (S\otimes_k\overline{k})^{G_{\overline{k}}}$
and this isomorphism respects the grading, which implies that the Molien series
does not change.
Thus, we may assume $k$ to be algebraically closed.

(1) If $p=0$, then $G$ is the constant group scheme associated to $G(k)$
and then, the claim becomes the classical Molien formula.

(2).(b) If $p>0$, then Proposition \ref{prop: lift} implies that the Molien series
of $S^G$ coincides with that of $K[x_1,...,x_n]^{\mathcal{G}_K}$
and then, claim (2) follows from the already established
characteristic zero case.

We have that $\dim_kS_i^G$ coincides with $\dim_k S^i(V)^G$.
Since the specialisation map $\mathrm{sp}$ respects tensor products, direct sums, etc., 
this dimension coincides with $\dim_\CC S^i(\CC^n)^{G_{\mathrm{abs}}}$.
Thus, claim (3) follows for the classical Molien formula associated to the linear
$G_{\mathrm{abs}}$-action on $\CC[x_1,...,x_n]$ defined by $\rho_\CC$.
\end{proof}

\begin{Remark}
\label{rem: molien}
In the case, where $G$ is a constant group scheme,
we have $G(k)=G(\overline{k})\cong G_{\mathrm{abs}}$ and then, 
we recover Molien's formula for finite groups of order prime to $p$.
Note however, that our version of Molien's formula also holds
if $G(k)$ is strictly smaller than $G_{\mathrm{abs}}$.
The following example illustrates this phenomenon.
\end{Remark}

\begin{Example}
\label{ex: molien}
Let $k$ be a field of characteristic $p\geq0$
and consider the group scheme $\bmu_n$, which is of length $n$ over $k$.
We view the standard embedding $\bmu_n\to\GG_{m,k}$
as a representation $\rho:\bmu_n\to\GL(V)$ with $\dim_kV=1$.
Then, $k[V]\cong k[x]$, $k[V]^{\bmu_n}\cong k[x^n]\subset k[x]$ 
and thus, we find
$$
\mathrm{Molien}\left(k[x]^{\bmu_n},t\right) \,=\,\frac{1}{1-t^n}.
$$
However, $\bmu_n(k)$ can even be trivial 
if $p>0$ or if $k$ is not algebraically closed.
For example, if $n=p>0$ or if $k\subseteq\RR$ and $n$ is odd, then $\bmu_n(k)=\{1\}$
Thus, the Molien series is usually not a function of the group $\bmu_n(k)$
of $k$-rational points of $\bmu_n$. 
\end{Example}

Our final example shows that, loosely speaking, having full control over
the graded parts of an invariant ring does not imply having full control over 
the beta-number.

\begin{Example}
 Let $R$ be a ring.
 The symmetric group $\mathfrak{S}_n$ on $n$ letters acts on 
 $V:=R^n$ by permuting the basis elements and the induced action on 
 $S:=S(V)\cong R[x_1,...,x_n]$ acts by permuting the variables.
  Let $G\subseteq\mathfrak{S}_n$ be a subgroup
  and set $\rho$ to be the $R$-linear representation $G\to\mathfrak{S}_n\to\GL(V)$.
  \begin{enumerate}
  \item If $R=k$ is a field of characteristic $p$, then the Molien series
  of $S^G\subseteq S$ does not depend on $p$, see \cite[Proposition 3.4.4]{DK}. 
  \item If $R=k$ is a field of characteristic $p$, then $\beta(S^G)$ 
  may depend on $p$, as the following example
  shows:
 Consider the special case $n=6$ and let $G$ be the subgroup of 
 $\mathfrak{S}_6$ that is generated by the transposition $(1,2)(3,4)(5,6)$, which is isomorphic to $\ZZ/2\ZZ$.
 Let $S^G\subseteq S$ be the invariant subring.
 \begin{enumerate}
 \item If $p\neq2$, then we have $\beta(G, \rho)=2$, which follows, for example,
 from the classical Noether bound.
 \item 
 In any characteristic, we have the invariants of degree $\leq2$
  $$
  \begin{array}{lll}
    f_1 = x_1 + x_2,&
    f_2 = x_3 + x_4,&
    f_3 = x_5 + x_6,\\
    f_4 = x_1\cdot x_2,&
    f_5 = x_3\cdot x_4,&
    f_6 = x_5\cdot x_6,\\
    f_7 = x_1\cdot x_3 + x_2\cdot x_4,&
    f_8 = x_1\cdot x_5 + x_2\cdot x_6,&
    f_9 = x_3\cdot x_5 + x_4\cdot x_6.
    \end{array}
  $$
  Consider the invariant $f=x_1x_4x_5+x_2x_3x_6$ and note that
  $$
  2f \,=\, f_1\cdot f_2\cdot f_3 - f_1\cdot f_9  + f_2\cdot f_8 - f_3\cdot f_7,
  $$
  that is, if $p\neq2$, then $f$ can be expressed in terms of the $f_1,...,f_9$.
 
  If $p=2$, then one can check that $f$ is not expressible by invariants of degree $\leq2$,
  which shows $\beta(G,\rho)\geq3$ in this case.
  \end{enumerate}
  \item If $k$ is an algebraically closed field of characteristic $p=2$ and 
  $R=W(k)$, we find that $S^G\subset S$ is a flat lift of 
  $k[V\otimes_Rk]^G\subset k[V\otimes_Rk]$ to $R$,
  compatible with the $G$-action.
  Let $K$ be the field of fractions of $W(k)$.
  Although the Molien series of $k[V]^G$ and $K[V]^G$ coincide, we have
  $\beta(K[V]^G)<\beta(k[V]^G)$, that is,
  $\beta$ may change under lifting.
  Note however, that in this case
  the (associated) finite group scheme $\mathcal{G}\to\Spec W(k)$ is
  not linearly reductive since its special fibre $G\to\Spec k$ is not linearly reductive
  (it is \'etale, but not of order prime to $p$).
   \end{enumerate}
  \end{Example}

\end{document}